\documentclass{article}
\usepackage{amsfonts}
\usepackage{amsmath}
\usepackage{amsthm}
\usepackage{url}
\usepackage{hyperref}
\usepackage[capitalise, noabbrev]{cleveref}

\newtheorem{theorem}{Theorem} 

\newtheorem{lemma}[theorem]{Lemma}
\newtheorem{corollary}[theorem]{Corollary}

\newtheorem{definition}[theorem]{Definition}

\newcommand{\F}{\mathbb{F}}

\DeclareMathOperator{\PG}{PG}
\DeclareMathOperator{\AG}{AG}
\newcommand\st{\,:\,}

\begin{document}
\date{March 16, 2024}
\title{Mutually orthogoval projective and affine spaces}
\author{Mark Saaltink}

\maketitle

\begin{abstract}
  A recent paper 
  showed how to find sets of finite affine or projective planes constructed on a common set of points, so that lines of one plane meet lines of a different plane in at most two points.  In this paper, those results are generalized in two different ways to spaces of higher dimension.  The simpler of the two generalizations admits many solutions, both affine and projective.  For the stronger definition, where a line of one space must be an arc in the other, we show the existence of pairs of projective spaces of dimension one less than a prime.
\end{abstract}

\section{Introduction}

It is an old result, provable in several different ways, that for every projective plane $\PG(2, \F_q)$ there exists a second projective plane on the same set of points, whose lines are \emph{ovals} of the first plane,
that is, lines of the first plane meet lines of the second in at most two points.
This result was recently extended~\cite{CIJSSS2022} to certain affine planes.
In that paper, the authors define ``orthogoval'' as follows:
\begin{quote}
A pair of planes, both projective or both affine, of the same order and on the same point set are \emph{orthogoval} if each line of one plane intersects each line of the other plane in at most two points.
\end{quote}
That paper goes on to exhibit pairs (and sometimes larger sets) of orthogoval planes over many different fields and gives a new construction for pairs of orthogoval projective planes.

Background information on finite geometries can be found in \cite{Beutelspacher1998,casse_projective_2006,dembowski1968,hirschfeld_projective_1998}.
We write $\F_q$ for the finite field with $q$ elements,
$\AG(d, K)$ for the affine space of dimension $d$ over field $K$, and
$\PG(d, K)$ for the projective space of dimension $d$ over field $K$.
The letter $q$ always denotes a prime power.

Specific results in~\cite{CIJSSS2022} include the following:
\begin{itemize}
\item For every $n$ there is a pair of orthogoval $\AG(2, \F_{2^n})$ (\cite[Corollary 3.9]{CIJSSS2022}); if $n$ is relatively prime to 6 then there is a set of three mutually orthogoval $\AG(2, \F_{2^n})$ (\cite[Corollary 3.16]{CIJSSS2022}).
\item There is a set of seven mutually orthogoval $\AG(2,\F_3)$ (\cite[Theorem 3.1]{CIJSSS2022}).
  No pairs of orthogoval affine planes of other odd orders are known.
\item There are sets of seven mutually orthogoval $\AG(2,\F_4)$ and seven mutually orthogoval $\AG(2,\F_8)$ (\cite[Theorem 3.19]{CIJSSS2022}).
\item For every prime power $q$ there is a pair of orthogoval $\PG(2,\F_q)$ (\cite[Theorem 2.2]{CIJSSS2022}).
  \item There is a set of four mutually orthogoval $\PG(2,\F_3)$, but no more than two mutually orthogoval $\PG(2,\F_q)$ for $q \in \{2,4,5\}$ (\cite[Theorem 2.3]{CIJSSS2022}).
\end{itemize}

An open question in that paper is whether this concept and these results can be generalized to higher-dimensional structures.
Here we answer that question in the affirmative for two different generalizations.  For the first, and simplest, of these, we just replace the word ``planes'' with ``spaces'' in the definition of orthogoval.  This gives the following generalization:
\begin{definition} \label{main-def}
  A pair of spaces, both projective or both affine, of the same dimension and order and on the same point set are \emph{orthogoval} if each line of one space intersects each line of the other space in at most two points.
\end{definition}
Using Segre's terminology, we can restate the definition as ``\dots if each line of one space is a cap in the other'', as a \emph{cap} is a set of points that meets any line in at most two points, a natural generalization of an oval in the plane.

This generalization is not vacuous; in this paper we show the existence of
\begin{itemize}
\item a pair of orthogoval $\AG(k, \F_{2^n})$ for any $k \geq 2$
(\cref{thm:ag_k_2n}),
  \item  sets of three mutually orthogoval $\AG(2k, \F_{2^n})$ when $n$ is relatively prime to 6 (\cref{thm:3-lo-ag2k-f2n}),
\item  sets of seven  mutually orthogoval $\AG(k, \F_3)$ (\cref{thm:lo-ag3-f3}),
  seven  mutually orthogoval $\AG(2k, \F_4)$ (\cref{thm:seven-f4-f8}), and
  seven  mutually orthogoval $\AG(2k, \F_8)$ (\cref{thm:seven-f4-f8}) for any $k \geq 1$,
\item  two orthogoval $\PG(2k, \F_q)$ for any $k$ and any prime power $q$ (\cref{thm:lo-pg2k-fq}),
  \item a set of 6 mutually orthogoval $\PG(4, \F_2)$, 
 a set of 18 mutually orthogoval $\PG(6, \F_2)$, 
 a set of 10 mutually orthogoval $\PG(4, \F_3)$, 
 a set of 78 mutually orthogoval $\PG(6, \F_3)$, 
 a set of 3 mutually orthogoval $\PG(4, \F_4)$, 
 a set of 11 mutually orthogoval $\PG(6, \F_4)$, and 
 a set of 7 mutually orthogoval $\PG(4, \F_5)$ (\cref{thm:big-sets-pg}),
  \item a pair of orthogoval $\PG(3,\F_3)$ (\cref{thm:lo-pg3-f3}),
  \item arbitrarily large sets of mutually orthogoval projective spaces (of high dimension) over any $\F_q$, unless $q+1$ is a power of 2
    (\cref{th:arb-large-pg}).
\end{itemize}

In \cref{sec:other-gen} we discuss two other possible generalizations of the notion of orthogoval.  Both are stronger than the above definition. We show a construction for one of them, which uses arcs instead of caps, in \cref{thm:hyperplane-ortho} but know of no construction for the other.  The definition based on arcs uses terminology that extends its applicability, as will be described in \cref{sec:askew}.
\begin{definition} \label{def:askew}
  A pair of dimension $k$ spaces, both projective or both affine, of the same order and on the same point set are \emph{askew} if the points of any line of one space are in general linear position in the other.
\end{definition}
Our main result for this definition is that if $k+1$ is prime then there is a pair of askew $\PG(k, \F_q)$ (\cref{thm:hyperplane-ortho}).

\section{Constructions of orthogoval affine spaces}

In this section we generalize the construction of \cite[Section 3]{CIJSSS2022} to give orthogoval
affine spaces of dimension larger than 2.
We start with a generalization of \cite[Lemma 3.2]{CIJSSS2022}.
\begin{lemma} \label{lemma:a-b-no-roots}
  Let $\F = \F_q$ be a finite field and $p \in \F[x]$ be any polynomial.  Then there are $a, b \in \F$ such that $p(x) + ax + b$ has no roots in $\F$.
\end{lemma}
\begin{proof} Consider a $q \times q$ array $C$ whose entries are sets, with $C_{a,b} = \{ x \in \F \st p(x) + ax +b = 0 \}$.  We will show that $C$ has some empty cells.
  Observe first that a given value of $x$ appears exactly once in the row for $a$, as $x \in C_{a,b}$ iff $b = -p(x) - ax$.
  So a given $x$ lies in exactly $q$ cells, and as there are $q$ available values of $x$ and $q^2$ cells, if all cells are inhabited, each cell can contain at most 1 element.
  Conversely, if any cell has more than one occupant, some cell must be empty.
  Pick any $x \not=y$
  and let $a =  -\frac {p(x) - p(y)} {x-y}$ and  $b = p(x) + ax = p(y) +ay$,
  so that $C_{a,-b}$ contains both $x$ and $y$.
\end{proof}

We will prove a slightly over-general result for any characteristic, requiring a further generalization of
the concept of orthogoval.
\begin{definition}
A pair of spaces, both projective or both affine, of the same order and on the same point set are \emph{$k$-orthogoval} if each line of one space intersects each line of the other space in at most $k$ points.
\end{definition}
For $k=2$ this is the same as orthogoval as defined above.

We can generalize the construction of~\cite[Section 3.1]{CIJSSS2022} with respect to both dimension and characteristic.
\begin{theorem} \label{thm:p-orth-char-p}
 For any prime $p$, integer $n \geq 1$,  and integer $k \geq 2$, there is a pair of $p$-orthogoval $\AG(k, \F_{p^n})$.
\end{theorem}
\begin{proof} Fix $n \geq 1$, $k \geq 2$, and prime $p$.
  Let $\F = \F_{p^n}$, and pick $A, B \in \F$ so that $x^{(p^{k}-1)/(p-1)} + Ax + B$ has no root in $\F$; we know they exist by \cref{lemma:a-b-no-roots}.
  Then define $f: \AG(k, \F) \rightarrow \AG(k, \F)$ by the equation
  \[ f(x_1, \dotsc, x_k) = (x_1^p - x_2, \dotsc, x_{k-1}^p - x_k, \, x_k^p + Ax_2 + Bx_1). \]
  This $f$ is $\F_p$-multilinear, and so is a bijection iff 0 has a single preimage.  But if $f(x_1, \dotsc, x_k) = 0$ we have $x_i = x_1^{p^{i-1}}$ for $i = 1, \dotsc, k$ and the final component is 0 iff $x_1^{p^{k-1}} + Ax_1^p + Bx_1 = 0$.  This implies $x_1 = 0$ by the choice of $A$ and $B$ (otherwise $y=x_1^{p-1}$ is a root of $y^{(p^{k}-1)/(p-1)} + Ay + B$, which is impossible).

  One space is $\AG(k, \F)$ in its standard representation, and the second, on the same point set, has as its lines the sets $f(l) = \{ f(x_1, \dotsc, x_k) \st (x_1, \dotsc, x_k) \in l \}$ as $l$ runs over the lines of the first space.
  We now need to look at the intersections of lines from the two spaces.  As $f$ is $\F_p$-linear, we can look at just the lines through the origin, as any pair of lines can be translated so that they meet at the origin.
  For the standard plane, the set of lines through the origin is
  \[ L = \{ ~\{ (a_1v, a_2v, \dotsc, a_kv) \st v \in \F \} \st (a_1:a_2:\cdots:a_k) \in \PG(k-1, \F) \}. \]
  For the second plane we have
  \begin{align*} L' &= \{ \{ f(b_1w, \dotsc, b_kw) \st w \in \F \} \st (b_1: \cdots : b_k) \in \PG(k-1, \F) \} \\
    &= \{~\{ (b_1^pw^p - b_2w, \dotsc, b_{k-1}^pw^p - b_kw,\,  b_k^pw^p+(Ab_2 + Bb_1)w) \} \st \dots \}. 
  \end{align*}
  Now fixing a particular $(a_1:\cdots:a_k)$ and $(b_1:\cdots:b_k)$, we will show that the corresponding lines meet in at most $p$ points.
  Where these lines meet we have
  \begin{align}
    a_i v &= b_i^p w^p - b_{i+1}w \qquad \text{for $1 \leq i \leq k-1$, and} \label{eq:a_i_1} \\
    a_k v &=  b_k^pw^p+(Ab_2 + Bb_1)w. \label{eq:a_i_2}
  \end{align}
  We distinguish two cases.

  {\bf Case 1: $a_1 \not= 0$}.  As the $a_i$ are chosen projectively we may assume $a_1 = 1$.  Then from \cref{eq:a_i_1} with $i=1$ we have
  $v = b_1^p w^p - b_2 w$; substituting this in the remaining equations gives
  \begin{align*}
    a_i ( b_1^p w^p - b_2 w) &= b_i^p w^p - b_{i+1}w \qquad \text{for $2 \leq i \leq k-1$, and} \\ 
    a_k ( b_1^p w^p - b_2 w) &=  b_k^pw^p+(Ab_2 + Bb_1)w. 
  \end{align*}
  Collecting terms and rearranging gives
  \begin{align*}
    (a_i b_1^p - b_i^p) w^p + (b_{i+1} - a_i b_2) w &= 0 \qquad \text{for $2 \leq i \leq k-1$, and} \\ 
    (a_k b_1^p - b_k^p) w^p - (a_k b_2 + Ab_2 + Bb_1) w &=  0. 
  \end{align*}
  If any of those equations considered as a polynomial in $w$  has a nonzero coefficient, there can be at most $p$ satisfying values of $w$ and so at most $p$ points of intersection.  So suppose the coefficients are all 0; this gives
  \begin{align}
    a_i b_1^p &= b_i^p & \text{for $2 \leq i \leq k$,} \label{eq:a_i_7}\\
    b_{i+1} &= a_i b_2 & \text{for $2 \leq i \leq k-1$, and} \label{eq:a_i_8}\\
    a_k b_2 + Ab_2 + Bb_1 &=  0. \label{eq:a_i_9}
  \end{align}
  From \cref{eq:a_i_7} we see that if $b_1 = 0$, then all the $b_i$ are 0, an impossibility.  So we may assume $b_1 = 1$, and \cref{eq:a_i_7} gives
  \begin{equation} \label{eq:a-from-b}
    a_i = b_i^p
  \end{equation}
  for all $i$, letting us eliminate the $a_i$.  Now \cref{eq:a_i_8} gives
  \[ b_{i+1} = b_i^p b_2, \]
so that induction gives, for $2 \leq i \leq k$,
\[ b_{i} = b_2^{1+p+\cdots+p^{i-2}} = b_2^{(p^{i-1}-1)/(p-1)}. \]
Substituting this expression for $b_k$ into \cref{eq:a-from-b} gives
\[ a_k = b_2^{p+p^2+\cdots+p^{k-1}}, \]
which we can substitute into
\cref{eq:a_i_9}, and recalling $b_1 = 1$, we get
\[ b_2^{(p^k-1)/(p-1)} + Ab_2 + B = 0. \]
which is impossible by the choice of $A$ and $B$.

{\bf Case 2: $a_1 = 0$.} Let $j$ be the lowest index so that $a_j \not=0$,
so that $2 \leq j \leq k$.  We may assume $a_j = 1$.
Then \cref{eq:a_i_1} for $1 \leq i < j$ is
\[ 0 = b_i^p w^p - b_{i+1}w, \]
giving at most $p$ choices for $w$ unless $b_i = b_{i+1} = 0$. So, assuming that for all $i < j$, we must have $j < k-1$, and \cref{eq:a_i_1} for $i=j$ with $b_j = 0$  gives
  \[ v = - b_{j+1}w. \]
  Substituting this into \cref{eq:a_i_1,eq:a_i_2} gives
    \begin{align*}
    - a_i b_{j+1} w &= b_i^p w^p - b_{i+1}w \qquad \text{for $j+1 \leq i \leq k-1$, and} \\
    - a_k  b_{j+1} w &=  b_k^pw^p+(Ab_2 + Bb_1)w.
    \end{align*}
    But then unless $b_i = 0$ for $j+1 \leq i \leq k$, there can be at most $p$ solutions for $w$.  As the $b_i$ cannot be all zero, this case is done.
\end{proof}

With $p=2$, this gives orthogoval affine spaces of any dimension larger than 1 over a field of even characteristic.
\begin{corollary} \label{thm:ag_k_2n}
For any $k \geq 2$ and $n \geq 1$ there is a pair of orthogoval $\AG(k, \F_{2^n})$.
\end{corollary}

For some fields we can do better than \cref{thm:p-orth-char-p} allows.  Over $\F_5$, for example, we can find a set of 24 mutually $3$-orthogoval affine planes (and do not know if this is the best possible), while the theorem promises only a pair of 5-orthogoval affine planes.

\begin{lemma} \label{thm:lo-add-dim}
  Let $\F$ be a field.  Suppose there exists a set of $k$ mutually orthogoval $\AG(m, \F)$ and a set of $k$ mutually orthogoval $\AG(n, \F)$.  Then there exists a set of $k$ mutually orthogoval $\AG(m+n, \F)$.
\end{lemma}
\begin{proof}
  For simplicity we show the proof for $k=2$, which can be seen to generalize as needed.
  Represent the points of $\AG(m, \F)$ in the usual way by $\F^m$ and those of $\AG(n, \F)$ by $\F^n$.
  Let $L$ be the set of lines in $\AG(m, \F)$ and $L'$ the set of lines in  $\AG(n, \F)$.  By assumption, we have bijections
  $f: \F^m \rightarrow  \F^m$ and $f': \F^n \rightarrow  \F^n$
  given by arbitrary isomorphisms between the spaces,
  so that
  \begin{align*}
    | f(l_1) \cap l_2 | \leq 2 \qquad & \text{for $l_1, l_2 \in L$, and} \\
    | f'(l_1) \cap l_2 | \leq 2 \qquad & \text{for $l_1, l_2 \in L'$.}
  \end{align*}
  We can represent $\AG(m+n, \F)$ with its points as the set $\F^m \times \F^n$.  Let $\pi_1$ and $\pi_2$ be the projection functions from $\F^m \times \F^n$ to $\F^m$ and to $\F^n$ respectively, and let $L''$ be the set of lines. Any line $l \in L''$ has the form
  \[ l = \{ (s,t) + k(u,v) \st k \in \F \} \]
  where $u \in \F^m$ and $v \in \F^n$ are not both 0, so we have either
  \begin{itemize}
  \item $v = 0$, so that $\pi_1(l) \in L$ and $\pi_2(l)$ is a singleton,
  \item $u = 0$, so that $\pi_1(l)$ is a singleton and $\pi_2(l) \in L'$, or
  \item $\pi_1(l) \in L$ and $\pi_2(l) \in L'$.
  \end{itemize}
    When
  $\pi_1(l) \in L$, then $\pi_1$ is one-to-one on $l$, and similarly when $\pi_2(l) \in L'$, then $\pi_2$ is one-to-one on $l$.

  Define $g: \F^m \times \F^n \rightarrow \F^m \times \F^n$ by the equation
  $g(x,y) = (f(x), f'(y))$.  This $g$ is a bijection. Now let $l_1, l_2 \in L''$; we must show that
  \[ | g(l_1) \cap l_2 | \leq 2.\]
  Suppose first that $\pi_1(l_1) \in L$. Then
  \[ \pi_1(g(l_1)) = f(\pi_1(l_1)) \]
  and as $\pi_1(l_2)$ is either an element of $L$ or a singleton,
  $\pi_1(g(l_1))$ and $\pi_1(l_2)$ meet in at most two points.
  The other case, $\pi_2(l_1) \in L'$, is similar.
\end{proof}

We can thus use \cite[Corollary 3.16]{CIJSSS2022} as a base case to get a set of 3 mutually orthogoval spaces of even dimension over some fields of characteristic 2
\begin{corollary} \label{thm:3-lo-ag2k-f2n}
  If $n$ is relatively prime to 6, there exists a set of 3 mutually orthogoval $\AG(2k, \F_{2^n})$.
\end{corollary}

We can also use the sets of 7 mutually orthogoval $\AG(2,\F_4)$ and 7 mutually orthogoval $\AG(2,\F_8)$ from \cite[Theorem 3.19]{CIJSSS2022}) as a base case to get sets of any even dimension over these fields:
\begin{corollary} \label{thm:seven-f4-f8}
  For any $k \geq 1$ there exists a set of  seven  mutually orthogoval $\AG(2k, \F_4)$, and a set of
  seven  mutually orthogoval $\AG(2k, \F_8)$.
\end{corollary}

The following was found with computer search.\footnote{This search was carried out by coding a recognizer for a solution as a functional program  in Cryptol (\url{www.cryptol.net}), which calls out to the SMT solver Z3~\cite{de2008z3} with a Boolean satisfiability problem whose solution is a value for which the program returns true.  The solution was then checked independently, in SageMath~\cite{sagemath}.}
\begin{theorem}
  There exists a set of eight mutually orthogoval $\AG(3, \F_3)$.
\end{theorem}
\begin{proof}
    We label the points of $\F_3^3$ with integers in the range 0 to 26, with point $(a,b,c)$ labelled as $9a+3b+c$.  The additive group on this set $\F_3^3$ adds pointwise, and the corresponding operation is used on the integer representation; thus the sum of 13 and 5 is 15, as $(1,1,1) + (0,1,2) = (1,2,0)$.
  Let $A_0$ be the usual affine space on this set, with lines the sets $\{P, Q, R\}$ with $P \not= Q$ and $P+Q+R=0$.  (This is the same as sets $\{P, P+S, P+2S \}$ for nonzero $S$)
  Let $\sigma$ be the permutation of $\{0, 1, \dotsc, 26 \}$ with cycle composition
  \[ (2, 24, 23, 10, 25, 7, 3, 5)(4, 6, 11, 15, 13, 12, 9, 19)(8, 21, 22, 16, 17, 18, 20, 26), \]
  thus with fixed points 0, 1, and 14.
  We let $A_i$ be the affine space whose lines are the images of the lines of $A_0$ through the function $\sigma^i$.
  For example, $\{0,1,2\}$ and $\{5, 13, 21\}$ are lines of $A_0$; the corresponding lines of $A_1$ are $\{0, 1, 24\}$ and $\{2, 12, 22\}$, and in $A_2$ we have $\{0, 1, 23\}$ and $\{24, 9, 16\}$.
  The spaces $\{A_0, A_1, \dotsc, A_7\}$ are mutually orthogoval.
\end{proof}

Together with the known set of seven mutually orthogoval $\AG(2, \F_3)$ and with Theorem~\ref{thm:lo-add-dim}, we obtain
\begin{corollary} \label{thm:lo-ag3-f3}
  For any $r \geq 2$, there exists a set of seven mutually orthogoval $\AG(r, \F_3)$.
\end{corollary}

\section{Projective spaces} \label{sec:lo-projective}

Many authors, for example ~\cite{jungnickel1984,MR3248524,sherk_1986}, have noted the use of inversion (or negation in the case of difference sets) to find orthogoval planes (although not using that terminology).
We can generalize this method
to construct orthogoval projective spaces of any even dimension over $\F_q$ for any prime power $q$, and can find other larger sets of orthogoval spaces in many other cases.

\begin{definition}
  A function $f: \PG(d, \F_q) \rightarrow \PG(d, \F_q)$ mapping the points of the projective space to itself is an \emph{orthomorphism} if it is a bijection and the image of any line is a cap.
\end{definition}
Recall that a \emph{cap} in a finite geometry is a set of points such that no three are on a line.  In the case of affine planes, such functions were called \emph{affine ovalinear} in~\cite{CIJSSS2022}.

For any bijection $f: \PG(d, \F_q) \rightarrow \PG(d, \F_q)$  mapping the points of the projective space to itself we can define a space $f(\PG(d, \F_q))$ isomorphic to the standard $\PG(d, \F_q)$, whose lines are the $f$-images of the lines on the standard plane.  Then we have the basic properties
\begin{lemma} \label{lemma:orthomorphism-basic}
  Let $S = \PG(r, \F_q)$ and suppose that $f$ is a bijection on the points of $S$.
  \begin{enumerate}
  \item $f$ is an orthomorphism iff $f(S)$ is orthogoval to $S$.
  \item If $g$ is a bijection then $f(S)$ is orthogoval to $S$ iff
    $(g \circ f)(S)$ is orthogoval to $g(S)$.
  \item $f$ is an orthomorphism iff its inverse is.
  \item if $f^i$ is an orthomorphism for $1 \leq i \leq n$ then the $n+1$ spaces
    $\{ f^i(S) \st 0 \leq i \leq n \}$ are mutually orthogoval.
  \end{enumerate}
\end{lemma}

\begin{definition} \label{def:Phi}
  Let $\F = \F_{q^{r}}$, let $S = \PG(r-1,\F_q)$, and let $z \in \F$ be some primitive element.
  Label the points of $S$ by nonzero elements of $\F$ as follows:
  any point $(x_1:x_2:\cdots:x_{r}) \in S$ is labelled with $\sum x_i z^i$.
  Every point therefore has $q-1$ different labels, each an $\F_q$-multiple of the other, and
  every nonzero element of $\F$ is a label of some point.
  For any integer $i$ relatively prime to $q^r-1$,
  the mapping $x \mapsto x^i$ on labels induces a function taking points of $S$ to points of $S$.  We call that mapping $\Phi_i$.
\end{definition}
  Under the conditions of the definition, $\Phi_i$ can be seen to induce a bijection on the points of $S$: we have some $k$ with $x^{ik} = x$ for all $x \in \F^*$ as $i$ is relatively prime to $q^r-1$.  So if $x^i$ and $y^i$ label the same point, then
  $x^i = a y^i$ for some $a \in \F_q^\ast$; thus $x = a^k y$, and the labels $x$ and $y$ are for the same point.

  Note that we have $\Phi_{ij} = \Phi_i \circ \Phi_j$ for any $i, j$.

\begin{theorem} \label{thm:lo-pg2k-fq}
For any $r \geq 1$ and prime power $q$, there is a pair of orthogoval $\PG(2r, \F_q)$.
\end{theorem}
\begin{proof}
  We form a second space on the points of $S = \PG(2r,\F_q)$ via the bijection $\Phi_{-1}$: let $S' = \Phi_{-1}(S)$.
  We will show that $S$ and $S'$ are orthogoval.
  To that end, let $l$, $m$, and $n$ label three colinear points of $S$.  Then there exists $a, b, c \in \F_q$ with
  $al + bm + cn = 0$.  These three points are colinear in $S'$ iff  we have some $d, e, f\in \F_q$ with
  $dl^{-1} + em^{-1} + fn^{-1} = 0$.  As $l$, $m$, and $n$  are the labels of distinct points, none of $a$, $b$, of $c$ is zero, and we may assume $a=1$; similarly we may assume $d=1$.  Then
  \[ 1 = l l^{-1} = ( bm + cn ) (em^{-1} + fn^{-1}) = (be+cf) + bfmn^{-1} + cem^{-1}n. \]
  Let $y = m^{-1}n$, collect terms, and multiply by $y$ to get
  \[ cey^2 + (be+cf-1)y + bf = 0. \]
  However, $ce$ is not 0, so $y \in \F$ satisfies a polynomial in $\F_q[x]$ of degree 2; therefore $y \in \F_{q} \cup \F_{q^2}$.
  As the  degree of $\F$ is odd, we must have $y \in \F_q$.  But then $m$ and $n$ are labels for the same point, contrary to assumption.
\end{proof}

Using a similar idea we can find other mappings that give pairs, and sometimes larger sets, of orthogoval spaces.

\begin{theorem} \label{thm:more-orthomorphisms}
Given a prime power $q = p^e$, integer $r$, and exponent $w\geq 2$ so that
$w$ is relatively prime to $q^r-1$, $r$ is relatively prime to $w!$, and $w$ is not a power of $p$,
then $\Phi_w$ is an orthomorphism on $\PG(r-1, \F_q)$.
\end{theorem}
\begin{proof}
We let the lines of space $S'$ be the pre-images of the lines of $S$ under this bijection, that is $S' = \Phi_{-w}(S)$.  We will show that $S$ and $S'$ are orthogoval; \cref{lemma:orthomorphism-basic} then completes the proof.

Let $l$, $m$, and $n$ label three colinear points of $S$.  Then there exists $a, b, c \in \F_q$ with
  $al - bm - cn = 0$.  These three points are colinear in $S'$ iff  we have some $d, e, f\in \F_q$ with
$dl^w - em^w - fn^w = 0$.  As $l$, $m$, and $n$  are the labels of distinct points, none of $a$, $b$, of $c$ is zero, and we may assume $a=1$; similarly we may assume $d=1$.  Then $l = bm + cn$ and we have
\[ \left(bm+cn\right)^w - em^w - fn^w = 0. \]
Dividing by $n^w$ and setting $x = m/n$ gives
\[ (bx+c)^w - ex^w - f = 0. \]
Since $w$ is not a power of $p$, the expansion of $(bx+c)^w$ has a nonzero term
besides $b^wx^w$ and $c^w$; that term cannot be cancelled
by either $ex^w$ or 
$f$.
So we have a nonzero polynomial of degree at most $w$ with root $x \in \F$.  But $x \in \F_{q^r}$, so the minimal polynomial for $x$ has degree dividing $r$.
Thus the degree must be 1, with $x \in \F_q$.  That makes $m$ and $n$ labels of the same point, contrary to assumption.
\end{proof}
Note that this lemma is vacuous if $r$ is even, as then there is no possible $w$; similarly for $r=3$ the hypotheses cannot be satisfied. For other odd $r$ we have instances, for example with $r=q=5$, $w$ must be relatively prime to 2, 11, and 17 and less than 5, so that $w=3$ and $w=4$ give orthomorphisms.

We can use this lemma to show the existence of arbitrarily large sets of mutually orthogoval spaces if we allow the dimension to get high enough.
\begin{corollary} \label{th:arb-large-pg}
  If $q$ is a prime power, $q+1$ is not a power of 2, and $n \geq 1$, there is some $r$ so that we can find a set of $n$ mutually orthogoval $\PG(r-1, \F_q)$.
\end{corollary}
\begin{proof}
  Let $s$ be an odd prime factor of $q+1$ and
  pick some prime $r$ greater than $s^n$.  Then as $q \equiv -1 \pmod s$ and $r$ is odd, $q^r \equiv -1 \pmod s$ and $\gcd(s^i, q^r-1)=\gcd(s^i,-2) = 1$ for any $i$.  \cref{thm:more-orthomorphisms}
 with $w=s^i$ then shows that
$\Phi_{s^i}$ is an orthomorphism for $i \in \{1, 2, \dotsc, n\}$, so by \cref{lemma:orthomorphism-basic}
  any pair of the induced planes are orthogoval.
\end{proof}

By computer search we can find large sets of mutually orthogoval spaces of much lower dimension than this corollary gives.
\begin{theorem} \label{thm:big-sets-pg}
  There exist sets of
  6 mutually orthogoval $\PG(4, \F_2)$,
  18 mutually orthogoval $\PG(6, \F_2)$,
  10 mutually orthogoval $\PG(4, \F_3)$,
  78 mutually orthogoval $\PG(6, \F_3)$,
  3 mutually orthogoval $\PG(4, \F_4)$,
  11 mutually orthogoval $\PG(6, \F_4)$, and
  7 mutually orthogoval $\PG(4, \F_5)$.
\end{theorem}
\begin{proof}
  For the following values, the sets
  $\{~\Phi_{w^i} \st 1 \leq i \leq n~\}$ are orthomorphisms, and any pair of the
   induced planes are orthogoval.  Together with the standard space this gives a set of $n+1$ spaces.
\begin{center}
  \begin{tabular}{rrl
      r}
  $q$ & $r$ & $w$ & $n$ \\ \hline
  2 & 5 & 3, 11, 13, or 17 & 5 \\
  2 & 7 & 3 or 7 & 17 \\
  \hline
  3 & 5 & 17, 19 & 9 \\
  3 & 7 & 25 & 77 \\
  \hline
  4 & 5 & 7 & 2 \\
  4 & 7 & 23 & 10 \\
  \hline
  5 & 5 & 3, 9 & 6 \\
\end{tabular}
\end{center}
\end{proof}

I only know a few examples of orthogoval projective spaces of odd dimension.
\begin{lemma}
  There is a set of seven mutually orthogoval $\PG(3, \F_2)$
\end{lemma}
\begin{proof}
  Number the points of $\PG(3,\F_2)$ as follows: with $z \in \F_{16}$ satisfying $z^5+z+1$, any point $(x_1:x_2:x_3:x_4) \in \PG(3,\F_2)$ can be identified with
$\sum x_i z^i$, which itself can be represented by its logarithm to base $z$, giving a bijection $f$ from $\{0, \dots, 14 \}$, with $f(i) = z^i$.
  The permutation with cycle structure \[(0,4,5,6,2,10,7)(1,11,8,12,14,9,13)\] (and fixed point 3) has order 7; all the images of the standard space under powers of this permutation are orthogoval to one another.
\end{proof}

\begin{lemma} \label{thm:lo-pg3-f3}
  There is a pair of orthogoval $\PG(3,\F_3)$.
\end{lemma}
\begin{proof}
 Number the points of $\PG(3,\F_3)$ as follows: with $z \in \F_{81}$ satisfying $z^4-z^3-1 = 0$, any point $P = (a:b:c:d)$ of $\PG(3, \F_3)$ can be identified with integer $i$, where $0 \leq i < 40$,
  and $z^i = az^3+bz^2 + cz+d$.  Then the permutation with cycle representation
  \begin{multline*}(0, 16, 10, 22, 4, 24, 37, 6, 18, 13, 21, 36, 28, 31, \\ 34, 32, 33, 2, 27, 9, 5, 17, 38, 23, 11, 15, 14, 12, \\ 8, 20, 35, 19, 25, 1, 3, 29)\,(7,26,30)
    \end{multline*}
  gives the second space.
\end{proof}

\section{Bounds} \label{sec:bounds}

Just as for orthogoval planes, we can find an upper bound for the number of members in a set of mutually orthogoval spaces by counting triples of points.  The key observation is that a set of three points can be colinear in at most one space in the set.

For affine spaces $\AG(d, \F_q)$, every line contains $q$ points, and there are $p=q^d$ points and $l = \binom p 2 / \binom q 2$ lines in each space.  Each space thus has $l \binom q 3$ colinear triples of points.  There are $\binom p 3$ sets of 3 points, so the number of spaces in a set of mutually orthogoval affine spaces is at most
\[ \frac {\binom p 3} {l \binom q 3} =
   \frac {\binom p 3 \binom q 2} {\binom p 2 \binom q 3} =
   \frac {p-2} {q - 2} =
   \frac {q^d-2} {q-2}. \]

For projective spaces $\PG(d, \F_q)$, every line contains $q+1$ points and there are $p'=\frac {q^{d+1} - 1}{q-1}$ points and $l' = \binom {p'} 2 / \binom {q+1} 2$ lines in each space.  Each space thus has $l' \binom {q+1} 3$ colinear triples of points.  There are $\binom {p'} 3$ sets of 3 points, so the number of spaces in a set of mutually orthogoval projective spaces is at most
\[ \frac {\binom {p'} 3} {l' \binom {q+1} 3} =
   \frac {\binom {p'} 3 \binom {q+1} 2} {\binom {p'} 2 \binom {q+1} 3} =
   \frac {{p'}-2} {q - 1} =
   \frac {q^{d+1} - 2q + 1} {(q-1)^2}. \]

A slightly lower bound can be derived from the Johnson bound on constant-weight codes~\cite{johnson1962new}, which also gives a bound on packing problems~\cite{mills_mullin_1992,stinson_packings_2007}.
That bound states that the number of subsets of size $b$ of a set $S$ of size $n$ such that no three
elements of $S$ are contained in more than one block is at most
   \[ \left\lfloor \frac n b \left\lfloor \frac {n-1} {b-1} \left\lfloor \frac {n-2} {b-2} \right\rfloor \right\rfloor \right\rfloor. \]
Applying this to the affine case, with the lines of all the spaces as blocks, and using the fact that each space contains $l = \frac {p(p-1)}{q(q-1)}$ lines, gives a bound of at most
\[ \frac {q(q-1)}{p(p-1)} \left\lfloor \frac p q \left\lfloor \frac {p-1} {q-1} \left\lfloor \frac {p-2} {q-2} \right\rfloor \right\rfloor \right\rfloor \]
  members in any set of mutually orthogoval affine spaces.  This can be seen (by dropping the floor operations) to be a bit smaller than the bound we derived above, but also is less easily handled algebraically.

The sets we can construct have in general many fewer elements than the bounds allow for, so either the bounds or the constructions can be improved.

\section{Other generalizations} \label{sec:other-gen}

The generalization of the notion of orthogoval to higher dimensions used above is not the only possibility.
In this section we will use the term \emph{line-orthogoval} to distinguish the notion of \cref{main-def} from the other notions considered here.

\subsection{Hyperplane orthogoval and askew spaces} \label{sec:askew}

Lines are hyperplanes of a 2-dimensional geometry, so we can consider replacing some or all occurrences of ``line'' in \cref{main-def} by ``hyperplane'' to have
\begin{quote}
  A pair of dimension $k$ spaces, both projective or both affine, of the same order and on the same point set are \emph{hyperplane-orthogoval} if each line of one space intersects each hyperplane of the other space in at most $k$ points.
\end{quote}
When $k=2$, hyperplane-orthogoval is the same as line-orthogoval.
Clearly $k$ is the smallest possible value for the ``at most $k$ points'' part of this definition whenever it is not larger than the number of points on a line,
as any $k$ points on a line of one space will determine a hyperplane of the other.
Whenever $k$ is larger than or equal to the number of points in a line, that is $q$ for $\AG(r, \F_q)$ and $q+1$ for $\PG(r, \F_q)$, the property is trivial.

A set of points in a space of dimension $k$ such that any $k+1$ of them span the space is known as an \emph{arc}.  Arcs were first studied by Segre and have been widely investigated since then~\cite{blokhuis1990,HIRSCHFELD201544}.
So an equivalent definition of hyperplane-orthogoval for projective spaces is that each line of one space is an arc of the other.
However, this definition is still trivial whenever $k$ is larger than the number of points in a line.
We instead have a definition below that is nontrivial in all cases.
Recall that a set $S$ of points in a finite geometry of dimension $d$ is in \emph{general linear position}
iff for any $T \subseteq S$, the points of $T$ span a subspace of dimension $\min(|T|-1, d)$.  That is (when $d$ is large enough) no three points are colinear, no four are coplanar, and so on; the points are as independent as possible given the dimension of the ambient space.

If two spaces $S_1$ and $S_2$ of dimension $k\geq 2$ are hyperplane-orthogoval and $k$ is not larger than the number of points on a line, the spaces must be line-orthogoval; if $l_1$ is a line of $S_1$ and $l_2$ is a line of $S_2$ and they meet in 3 or more points, then picking $k-2$ extra points on $l_1$ determines at least one hyperplane of $S_2$ containing $k+1$ or more points of $l_1$.  By similar reasoning, a line of one space cannot contain any more than $j+1$ points of any subspace of dimension $j$ of the other space.
So, when $k$ is less than the number of points in a line and $S_1$ and $S_2$ are hyperplane-orthogoval,
the points of a line in one space are in general linear position in the second space.  This leads to our improved definition that is not trivial for small $q$ (shown in the introduction as \cref{def:askew}):

\begin{quote}
A pair of dimension $k$ spaces, both projective or both affine, of the same order and on the same point set are \emph{askew} if the points of any line of one space are in general linear position in the other.
\end{quote}

The inversion mapping of \cref{sec:lo-projective} produces projective spaces that are askew if the dimension is one less than a prime:
\begin{theorem} \label{thm:hyperplane-ortho}
If $k+1$ is prime and $q$ is a prime power, then there is a pair of askew $\PG(k, \F_q)$.
\end{theorem}
\begin{proof}
  We label the points of $S = \PG(k,\F_q)$ as in \cref{def:Phi}, with $z$ some primitive element of $\F_{q^{k+1}}$, and use $\Phi_{-1}$ to create a second space $T$ on the same set of points.  As $\Phi_{-1}$ is an involution we need only show that lines of $T$
  are in general position in $S$, and by symmetry can conclude that the lines of $S$ are in general position in $T$.

  For any $m$ with $2 \leq m \leq k$ consider some $m+1$ points $p_0, p_1, \dotsc, p_m$ in $S$
that are \emph{not} in general position.  Suppose further that we have
$p_i = (x_{i,0}: \cdots: x_{i,k})$.
Then these points lie in at least $l=k+1-m$ linearly independent hyperplanes $H_1, \dotsc, H_l$.
For each $i$ there are elements $c_{i,0}, \dotsc, c_{i,k} \in \F_q$, not all zero, so that
  \[ H_i = \left\{ (x_0 : \dots : x_k) \st \sum_{n=0}^k c_{i,n}x_n = 0 \right\}. \]
  If we form the $(m+1) \times (k+1)$ matrix
  \[ M = \begin{pmatrix} x_{0,0} & x_{0,1} & \dots & x_{0,k} \\
    x_{1,0} & x_{1,1} & \dots & x_{1,k} \\
    \vdots & \vdots & & \vdots \\
    x_{m,0} & x_{m,1} & \dots & x_{m,k} \\
  \end{pmatrix} \]
  we have \[ M (c_{i,0}, c_{i,1}, \dotsc, c_{i,k}) = 0 \]
  for each $i$,
  so $M$ has rank at most $m$.
  Thus there is some other nonzero vector $d = (d_0, d_1, \dotsc, d_k)$ with $dM = 0$.
  Noting now that $M (z^0, z^i, \dotsc, z^k)$ is a vector of labels for points $p_0, \dotsc, p_k$, we have that the $m+1$
  points are in general position
  iff their labels are linearly dependent over $\F_q$.  It is clear that this does not depend on which labels we choose.

  Suppose now some line $l$ of $T$ is not in general position in $S$.
  Then there are points $p_0, p_1, \dotsc, p_m$ in $S$ that are not in general position, with $m \leq k+1$.
  This line is the image of a line of $T$ through $\Phi_{-1}$; so let $P$ and $Q$ be two distinct points on that line and there must exist $a_i, b_i \in \F_q$ so that $\Phi_{-1}(p_i) = a_i P + b_i Q$.  Considering
  $P$ and $Q$ as labels, the label of $p_i$ is then $(a_i P + b_i Q)^{-1}$.  Now we have
  \[ \sum_{i=0}^m d_i (a_i P + b_i Q)^{-1} = 0.\]
  If we put $R = P/Q$ this gives
  \[ \sum_{i=0}^m d_i (a_i R + b_i)^{-1} = 0,\]
  and we can clear fractions to get
  \[ \sum_{i=0}^m d_i \prod_{j\not= i} (a_j R + b_j) = 0.\]
  Thus $R \in \F_{q^{k+1}}$ is the root of a polynomial of degree at most $m-1 \leq k$, namely
  \[ f(x) = \sum_{i=0}^m d_i \prod_{j\not= i} (a_j x + b_i) = 0.\]
  However, as $k+1$ is prime, the degree of $R$ can only be 1 or $k+1$, so if $f$ is nonzero it must be of degree 1.  That puts $R \in \F_q$ which is impossible as that would make $P$ and $Q$ labels for the same point.  So $f(x) = 0$ is the only remaining possibility.  This leads to a contradiction;  when $a_i \not= 0$ we would have
  \[ 0 = f(-b_i/a_i) = d_i \prod_{j\not= i} (a_j b_i/a_i - b_j), \]
  and as we must have $a_i b_j - a_jb_i \not= 0$ (so that $p_i$ and $p_j$ are distinct points) we have $d_i = 0$.  There can be at most one $a_i = 0$, if we suppose that this is $i=0$
  we now have\[ f(x) = d_0 \prod_{j\not= 0} (a_j x + b_j) \] (as all $d_i$ with $i \not= 0$ must be 0) and so $d_0 = 0$.  All the $d_i$ are zero; a contradiction.
\end{proof}

No other instances of askew spaces of dimension larger than two are known to me.

\subsection{Half-dimension orthogoval}
There is another way to generalize.  For orthogoval planes, the spaces are of dimension 2 and the lines of dimension 1---exactly half.  So perhaps that is the relationship to focus on.
\begin{definition}
  A pair of $2k$-dimensional spaces, both projective or both affine, of the same order and on the same point set are \emph{half-dimension-orthogoval} if each $k$-subspace of one space intersects each $k$-subspace of the other space in at most $k+1$ points.
\end{definition}
For $k=1$ this is the same as orthogoval.

The smallest non-trivial case for this would be $\AG(4, \F_2)$; each plane has 4 points so there is a real condition.  An exhaustive computer search shows that there does not exist a pair of half-dimension-orthogoval $\AG(4, \F_2)$.

If spaces $S$ and $T$ are half-dimension orthogoval, then they must also be line-orthogoval: suppose a line of $S$ and a line of $T$ meet in 3 or more points.  Then by adding $k-1$ more points we can get a set of $k+2$ points that must lie in subspaces of dimension $k$ of both $S$ and $T$, making them not half-dimension orthogoval.

Given \cref{thm:hyperplane-ortho} we might hope that $\Phi_{-1}$ would sometimes give half-dimension orthogoval projective spaces, but some computer exploration has not found any other than the line-orthogoval spaces for $k=1$.  It is unknown whether there are any half-dimension-orthogoval spaces of dimension larger then 2.

\section{Conclusions}

We have shown the existence
of many sets of spaces satisfying the generalized definition of orthogoval, \cref{main-def}, and have shown the existence of pairs of projective spaces of certain dimensions that satisfy the stronger \cref{def:askew}.
Much exploration remains to be done, and many questions are open:
\begin{itemize}
\item Can we find pairs of orthogoval affine spaces over fields of odd characteristic larger than 3?
  \item Are there orthogoval projective spaces of odd dimension besides those of \cref{thm:lo-pg3-f3} over $\F_3$?
  \item Is there a pair of orthogoval $\PG(r,\F_2)$ for all $r \geq 2$?
  \item Can the upper bounds of \cref{sec:bounds} be improved?
  \item Can we find askew affine spaces of dimension larger than 2?
  \item Can we find askew projective spaces of dimensions not covered by \cref{thm:hyperplane-ortho}?
  \item Do there exist any half-dimension orthogoval spaces of dimension larger than 2?
\end{itemize}

\bibliographystyle{plain}
\bibliography{higher-dimensions}

\end{document}